\documentclass[12pt]{amsart}%
\usepackage{graphicx}
\usepackage[small,nohug,heads=littlevee]{diagrams}
\usepackage{amscd}
\usepackage{geometry}
\usepackage{amsmath}
\usepackage{amsfonts}
\usepackage{amssymb}%
\newtheorem{theorem}{Theorem}[section]
\theoremstyle{plain}

\newtheorem{lemma}[theorem]{Lemma}

\newtheorem{proposition}[theorem]{Proposition}
\newtheorem{remark}{Remark}

\numberwithin{equation}{section}

\begin{document}
\title[Compacta with Shapes of Finite Complexes]{Compacta with Shapes of Finite Complexes: a direct approach to the
Edwards-Geoghegan-Wall obstruction}
\author{Craig R. Guilbault}
\address{Department of Mathematical Sciences\\
University of Wisconsin-Milwaukee, Milwaukee, WI 53201}
\email{craigg@uwm.edu}
\thanks{Work on this project was aided by a Simons Foundation Collaboration Grant.}
\date{August 21, 2015}

\begin{abstract}
An important \textquotedblleft stability\textquotedblright\ theorem in shape
theory, due to D.A. Edwards and R. Geoghegan, characterizes those compacta
having the same shape as a finite CW complex. In this note we present
straightforward and self-contained proof of that theorem.

\end{abstract}
\maketitle

\section{Introduction}

Before Ross Geoghegan turned his attention to the main topic of these
proceedings, \emph{Topological Methods in Group Theory}, he was a leader in
the area of shape theory. In fact, much of his pioneering work in geometric
group theory has involved taking key ideas from shape theory and recasting
them in the service of groups. Some of his early thoughts on that point of
view are captured nicely in \cite{Ge2}. Among the interesting ideas found in
that 1986 paper is an early recognition that a group boundary is well-defined
up to shape---an idea later formalized by Bestvina in \cite{Be}.

In this paper we return to the subject of Geoghegan's early work. For those
whose interests lie primarily in group theory, the work presented here
contains a concise and fairly gentle introduction to the ideas of shape
theory, via a careful study of one of its foundational questions.\medskip

In the 1970's D.A. Edwards and R. Geoghegan solved two open problems in shape
theory---both related to the issue of \textquotedblleft
stability\textquotedblright. Roughly speaking, these problems ask when a
\textquotedblleft bad\textquotedblright\ space has the same shape as a
\textquotedblleft good\textquotedblright\ space. For simplicity, we focus on
the following versions of those problems:\bigskip

\noindent\textbf{Problem A.} \emph{Give necessary and sufficient conditions
for a connected finite-dim\-en\-sion\-al compactum }$Z$\emph{\ to have the
pointed shape of a CW complex.\bigskip}

\noindent\textbf{Problem B.} \emph{Give necessary and sufficient conditions
for a connected finite-dim\-en\-sion\-al compactum }$Z$\emph{\ to have the
pointed shape of a finite CW complex.\bigskip}

Solutions to these problems can be found in the sequence of papers:
\cite{EG1}, \cite{EG2}, \cite{EG3}. A pair of particularly nice versions of
those solutions are as follows:\bigskip

\noindent\textbf{Solution A.} $Z$\emph{\ has the pointed shape of a CW
complex\ if and only if each of its homotopy pro-groups is stable.\bigskip}

\noindent\textbf{Solution B.} $Z$\emph{\ has the pointed shape of a finite CW
complex\ if and only if each of its homotopy pro-groups is stable and an
intrinsically defined Wall obstruction }$\omega\left(  Z,z\right)
\in\widetilde{K}_{0}\left(
\mathbb{Z}
\mathbb{[}\check{\pi}_{1}\left(  Z,z\right)  ]\right)  $
\emph{vanishes.\bigskip}

Solution B was obtained by combining Solution A with C.T.C. Wall's famous work
on finite homotopy types \cite{Wa}. So, in order to understand Edwards and
Geoghegan's solution to Problem B, it is necessary to understand two things:
Solution A; and Wall's work on the finiteness obstruction. Since both tasks
are substantial---and since Problem B can arise quite naturally without
regards to Problem A---we became interested in finding a simpler and more
direct solution to Problem B. This note contains such a solution. This paper
may be viewed as a sequel to \cite{Ge1}, where Geoghegan presented a new and
more elementary solution to Problem A. In the same spirit, we feel that our
work offers a simplified view of Problem B.

The strategy we use in attacking Problem B is straightforward and very
natural. Given a connected $n$-dimensional pointed compactum $Z$, begin with
an inverse system $K_{0}\overset{f_{1}}{\longleftarrow}K_{1}\overset{f_{2}%
}{\longleftarrow}K_{2}\overset{f_{3}}{\longleftarrow}\cdots$ of finite
$n$-dimensional (pointed) complexes with (pointed) cellular bonding maps that
represents $Z$. Under the assumption that $\operatorname*{pro}$-$\pi_{k}$ is
stable for all $k$, we borrow a technique from \cite{Fe1} allowing us to
attach cells to the $K_{i}$'s so that the bonding maps induce $\pi_{k}%
$-isomorphisms for increasingly large $k$. Our goal then is to reach a
\emph{finite} stage where the bonding maps induce $\pi_{k}$-isomorphisms for
all $k$, and are therefore homotopy equivalences. This would imply that $Z$
has the shape of any of those homotopy equivalent finite complexes. As
expected, we confront an obstruction lying in the reduced projective class
group of $\operatorname*{pro}$-$\pi_{1}$. Instead of invoking theorems from
\cite{Wa}, we uncover this obstruction in the natural context of the problem
at hand; in fact, the main result of \cite{Wa} can then be obtained as a
corollary. Another advantage to the approach taken here is that all CW
complexes used in this paper are finite. This makes both the algebra and the
shape theory more elementary.

\section{Background}

In this section we provide some background information on inverse systems,
inverse sequences, and shape theory. In addition, we will review the
definition of a reduced projective class group. A more complete treatment of
inverse systems and sequences can be found in \cite{Ge3}; an expanded version
of this introduction can be found in \cite{Gu}

\subsection{Inverse systems}

We provide a brief discussion of general inverse systems and pro-categories,
which provide the broad framework for more concrete constructions that will
follow. A thorough treatment of this topic can be found in \cite[Ch.11]{Ge3}.

An \emph{inverse system} $\left\{  X_{\alpha},f_{\alpha}^{\beta}%
;\mathcal{A}\right\}  _{\alpha\in\mathcal{A}}$ consists of a collection of
objects $X_{\alpha}$ from a category $\mathcal{C}$ indexed by a \emph{directed
set} $\mathcal{A}$, along with morphisms $f_{\alpha}^{\beta}:X_{\beta
}\rightarrow X_{\alpha}$ for every pair $\alpha,\beta\in\mathcal{A}$ with
$\alpha\leq\beta$, satisfying the property that $f_{\alpha}^{\gamma}=f_{\beta
}^{\gamma}\circ f_{\alpha}^{\beta}$ whenever $\alpha\leq\beta\leq\gamma$. By
fixing $\mathcal{C}$, but allowing the directed set to vary, and formulating
an appropriate definition of morphisms, one obtains a category
$\operatorname*{pro}$-$\mathcal{C}$ whose objects are all such inverse
systems. When $\mathcal{A}^{\prime}\subseteq\mathcal{A}$ is a directed set
there is an obvious subsystem $\left\{  X_{\alpha},f_{\alpha}^{\beta
};\mathcal{A}^{\prime}\right\}  _{\alpha\in\mathcal{A}^{\prime}}$ and an
inclusion morphism. When $\mathcal{A}^{\prime}$ is \emph{cofinal }in\emph{
}$\mathcal{A}$ (for every $\alpha\in\mathcal{A}$ there exists $\alpha^{\prime
}\in\mathcal{A}^{\prime}$ such that $\alpha\leq\alpha^{\prime}$), the
inclusion morphism is an isomorphism in $\operatorname*{pro}$-$\mathcal{C}$. A
key theme in this subject is that, when $\mathcal{A}^{\prime}$ is cofinal, the
corresponding subsystem contains all relevant information.

When $\mathcal{C}$ is a category of sets and functions, we may define the
\emph{inverse limit} of $\left\{  X_{\alpha},f_{\alpha}^{\beta};\mathcal{A}%
\right\}  _{\alpha\in\mathcal{A}}$ by%
\[
\underleftarrow{\lim}\left\{  X_{\alpha},f_{\alpha}^{\beta};\mathcal{A}%
\right\}  =\left\{  \left.  \left(  x_{\alpha}\right)  \in\prod_{\alpha
\in\mathcal{A}}X_{\alpha}\right\vert f_{\alpha}^{\beta}\left(  x_{\beta
}\right)  =x_{\alpha}\text{ for all }\alpha\leq\beta\right\}
\]
along with projections $p_{\alpha}:\underleftarrow{\lim}\left\{  X_{\alpha
},f_{\alpha}^{\beta};\mathcal{A}\right\}  \rightarrow X_{\alpha}$. When
$\mathcal{C}$ is made up of topological spaces and maps, the inverse limits
are topological spaces and the projections are continuous. Similarly,
additional structure is passed along to inverse limits when $\mathcal{C}$
consists of groups, rings, or modules and corresponding homomorphisms. An
important example of the \textquotedblleft key theme\textquotedblright\ noted
in the previous paragraph is that, when $\mathcal{A}^{\prime}$ is is cofinal
in $\mathcal{A}$, the canonical inclusion $\underleftarrow{\lim}\left\{
X_{\alpha},f_{\alpha}^{\beta};\mathcal{A}^{\prime}\right\}  \rightarrow
\underleftarrow{\lim}\left\{  X_{\alpha},f_{\alpha}^{\beta};\mathcal{A}%
^{\prime}\right\}  $ is a bijection of sets [resp., homeomorphism of spaces,
isomorphism of groups, etc.].

An \emph{inverse sequence} (or \emph{tower}) is an inverse system for which
$\mathcal{A}=%
\mathbb{N}
$, the natural numbers. Since all inverse systems used in this paper contain
cofinal inverse sequences, we are able to work almost entirely with towers.
General inverse systems play a useful, but mostly invisible, background role.

\subsection{Inverse sequences (aka towers)}

The fundamental notions that make up a category $\operatorname*{pro}%
$-$\mathcal{C}$ are simpler and more intuitive when restricted the the
subcategory of towers in $\mathcal{C}$. For our purposes, an understanding of
towers will suffice; so that is where we focus our attention.

Let
\[
C_{0}\overset{\lambda_{1}}{\longleftarrow}C_{1}\overset{\lambda_{2}%
}{\longleftarrow}C_{2}\overset{\lambda_{3}}{\longleftarrow}\cdots
\]
be an inverse sequence in $\operatorname*{pro}$-$\mathcal{C}$. A
\emph{subsequence} of $\left\{  C_{i},\lambda_{i}\right\}  $ is an inverse
sequence of the form%

\[
\begin{diagram}
C_{i_{0}} & \lTo^{\lambda_{i_{0}+1}\circ\cdots\circ\lambda_{i_{1}}
} & C_{i_{1}} & \lTo^{\lambda_{i_{1}+1}\circ\cdots\circ
\lambda_{i_{2}}} & C_{i_{2}} & \lTo^{\lambda_{i_{2}+1}\circ
\cdots\circ\lambda_{i_{3}}} & \cdots.
\end{diagram}
\]

\noindent In the future we will denote a composition $\lambda_{i}\circ
\cdots\circ\lambda_{j}$ ($i\leq j$) by $\lambda_{i,j}$.

\begin{remark}
\emph{Using the notation introduced in the previous subsection, a bonding map
}$\lambda_{i}$\emph{ would be labeled }$\lambda_{i-1}^{i}$\emph{ and a
composition }$\lambda_{i}\circ\cdots\circ\lambda_{j}$\emph{ (}$i\leq j$\emph{)
by }$\lambda_{i-1}^{j}$\emph{. When working with inverse sequences, we opt for
the slightly simpler notation described here.}
\end{remark}

Inverse sequences $\left\{  C_{i},\lambda_{i}\right\}  $ and $\left\{
D_{i},\mu_{i}\right\}  $ are isomorphic in $\operatorname*{pro}$-$\mathcal{C}%
$, or \emph{pro-isomorphic}, if after passing to subsequences, there exists a
commuting \emph{ladder diagram}:%
\begin{equation}
\begin{diagram} C_{i_{0}} & & \lTo^{\lambda_{i_{0}+1,i_{1}}} & & C_{i_{1}} & & \lTo^{\lambda_{i_{1}+1,i_{2}}} & & C_{i_{2}} & & \lTo^{\lambda_{i_{2}+1,i_{3}}}& & C_{i_{3}}& \cdots\\ & \luTo & & \ldTo & & \luTo & & \ldTo & & \luTo & & \ldTo &\\ & & D_{j_{0}} & & \lTo^{\mu_{j_{0}+1,j_{1}}} & & D_{j_{1}} & & \lTo^{\mu_{j_{1}+1,j_{2}}}& & D_{j_{2}} & & \lTo^{\mu_{j_{2}+1,j_{3}}} & & \cdots \end{diagram} \label{ladder diagram}%
\end{equation}
where the up and down arrows represent\ morphisms from $\mathcal{C}$. Clearly
an inverse sequence is pro-isomorphic to any of its subsequences. To avoid
tedious notation, we frequently do not distinguish $\left\{  C_{i},\lambda
_{i}\right\}  $ from its subsequences. Instead we simply assume that $\left\{
C_{i},\lambda_{i}\right\}  $ has the desired properties of a preferred
subsequence---often prefaced by the words \textquotedblleft after passing to a
subsequence and relabelling\textquotedblright.

\begin{remark}
\emph{Together the collection of down arrows in (\ref{ladder diagram})
determine a morphism in }$\operatorname*{pro}$\emph{-}$C$\emph{ from
}$\left\{  C_{i},\lambda_{i}\right\}  $\emph{ to }$\left\{  D_{i},\mu
_{i}\right\}  $\emph{ and the up arrows a morphism from }$\left\{  D_{i}%
,\mu_{i}\right\}  $\emph{ to }$\left\{  C_{i},\lambda_{i}\right\}  $\emph{.
Again see \cite[Ch.11]{Ge3} for details.}
\end{remark}

An inverse sequence $\left\{  C_{i},\lambda_{i}\right\}  $ is \emph{stable} if
it is pro-isomorphic to a constant sequence
\[
D\overset{\operatorname*{id}}{\longleftarrow}D\overset{\operatorname*{id}%
}{\longleftarrow}D\overset{\operatorname*{id}}{\longleftarrow}\cdots.
\]
For example, if each $\lambda_{i}$ is an isomorphism from $\mathcal{C}$, it is
easy to show that $\left\{  C_{i},\lambda_{i}\right\}  $ is stable.

Inverse limits of an inverse sequences of sets are particularly easy to
understand. In particular,
\[
\underleftarrow{\lim}\left\{  C_{i},\lambda_{i}\right\}  =\left\{  \left.
\left(  c_{0},c_{1},c_{2},\cdots\right)  \in\prod_{i=0}^{\infty}%
C_{i}\right\vert \lambda_{i}\left(  c_{i}\right)  =c_{i-1}\text{ for all
}i\geq1\right\}  ,
\]
with a projection map $p_{i}:\underleftarrow{\lim}\left\{  C_{i},\lambda
_{i}\right\}  \rightarrow C_{i}$ for each $i\geq0.$

\subsection{Inverse sequences of groups}

Of particular interest to us is the category $\mathcal{G}$ of groups and group
homomorphisms. It is easy to show that an inverse sequence of groups $\left\{
G_{i},\lambda_{i}\right\}  $ is stable if and only if, after passing to a
subsequence and relabelling, there is a commutative diagram of the form
\[
\begin{diagram}
G_{0}& & \lTo^{{\lambda}_{1}} & & G_{1} & & \lTo^{{\lambda}_{2}} & &
G_{2} & & \lTo^{{\lambda}_{3}} & & G_{3} &\cdots\\
& \luTo & & \ldTo & & \luTo & & \ldTo   & & \luTo & & \ldTo & \\
& & \operatorname{Im}\left(  \lambda_{1}\right) & & \lTo^{\cong} & & \operatorname{Im}\left(  \lambda
_{2}\right) & &\lTo^{\cong} & & \operatorname{Im}\left(  \lambda_{3}\right) & & \lTo^{\cong} & &\cdots & \\
\end{diagram}
\]
where all unlabeled maps are inclusions or restrictions. In this case
$\underleftarrow{\lim}\left\{  C_{i},\lambda_{i}\right\}  \cong%
\operatorname*{im}(\lambda_{i})$ and each projection homomorphism takes
$\underleftarrow{\lim}\left\{  C_{i},\lambda_{i}\right\}  $ isomorphically
onto the corresponding $\operatorname*{im}(\lambda_{i})$.

The sequence $\left\{  G_{i},\lambda_{i}\right\}  $ is \emph{semistable} (or
\emph{Mittag-Leffler }or \emph{pro-epimorphic}) if it is pro-isomorphic to an
inverse sequence $\left\{  H_{i},\mu_{i}\right\}  $ for which each $\mu_{i}$
is surjective. Equivalently, $\left\{  G_{i},\lambda_{i}\right\}  $ is
semistable if, after passing to a subsequence and relabelling, there is a
commutative diagram of the form
\[
\begin{diagram}
G_{0}& & \lTo^{{\lambda}_{1}} & & G_{1} & & \lTo^{{\lambda}_{2}} & &
G_{2} & & \lTo^{{\lambda}_{3}} & & G_{3} &\cdots\\
& \luTo & & \ldTo & & \luTo & & \ldTo   & & \luTo & & \ldTo & \\
& & \operatorname{Im}\left(  \lambda_{1}\right) & & \lOnto & & \operatorname{Im}\left(  \lambda
_{2}\right) & &\lOnto & & \operatorname{Im}\left(  \lambda_{3}\right) & & \lOnto & &\cdots & \\
\end{diagram}
\]
where \textquotedblleft$\twoheadleftarrow$\textquotedblright\ denotes a surjection.

\subsection{Inverse systems and sequences of CW complexes}

Another category of utmost interest to us is $\mathcal{FH}_{0}$, the category
of pointed, connected, finite CW complexes and pointed homotopy classes of
maps. (A space is \emph{pointed} if a basepoint has been chosen; a map is
\emph{pointed} if basepoint is taken to basepoint.) We will frequently refer
to pointed spaces and maps without explicitly mentioning the basepoints. We
will refer to an inverse system [resp., tower] from $\mathcal{FH}_{0}$ as an
\emph{inverse system} [resp., \emph{tower}]\emph{ of finite complexes.}

For each $k\geq1$, there is an obvious functor from $\operatorname*{pro}%
$-$\mathcal{FH}_{0}$ to $\operatorname*{pro}$-$\mathcal{G}$ taking an inverse
system $\left\{  K_{\alpha},g_{\alpha}^{\beta};\Omega\right\}  $ of pointed,
connected, finite simplicial complexes to be the inverse system of groups
$\left\{  \pi_{k}(K_{\alpha}),(g_{\alpha}^{\beta})_{\ast};\Omega\right\}  $
(the $k^{\text{th}}$ \emph{homotopy pro-group} \emph{of }$\left\{  K_{\alpha
},g_{\alpha}^{\beta};\Omega\right\}  $). A related functor takes $\left\{
K_{\alpha},g_{\alpha}^{\beta};\Omega\right\}  $ to the group $\underleftarrow
{\lim}\left\{  \pi_{k}(K_{\alpha}),(g_{\alpha}^{\beta})_{\ast};\Omega\right\}
$ which we denote $\check{\pi}_{k}\left(  \left\{  \pi_{k}(K_{\alpha
}),(g_{\alpha}^{\beta})_{\ast};\Omega\right\}  \right)  $ (the $k^{\text{th}}$
\emph{\v{C}ech homotopy group} \emph{of }$\left\{  K_{\alpha},g_{\alpha
}^{\beta};\Omega\right\}  $).

Clearly the initial functor described above takes towers from
$\operatorname*{pro}$-$\mathcal{FH}_{0}$ to towers in $\operatorname*{pro}%
$-$\mathcal{G}$, while the latter takes each tower to a group.

\subsection{Homotopy dimension}

The \emph{dimension}, $\dim(\left\{  K_{i},f_{i}\right\}  )$, of a tower of
finite complexes is the supremum (possibly $\infty$) of the dimensions of the
$K_{i}$'s. The \emph{homotopy dimension }of $\left\{  K_{i},f_{i}\right\}  $
is defined by:
\[
\hom\dim\left(  \left\{  K_{i},f_{i}\right\}  \right)  =\inf\{\left.
\dim(\left\{  L_{i},g_{i}\right\}  )\right\vert \left\{  L_{i},g_{i}\right\}
\text{ is pro-isomorphic to }\left\{  K_{i},f_{i}\right\}  \}.
\]

\subsection{Shapes of compacta\label{Subsection: shapes of compacta}}

Our view of \emph{shape theory }is that it is the study of (possibly bad)
compact metric spaces through the use of associated inverse systems and
sequences of finite complexes.

Let $Z$ be a compact, connected, metric space with basepoint $z$. Let $\Omega$
denote the set of all finite open covers $\mathcal{U}_{\alpha}$ of $Z$, each
with a distinguished element $U^{\ast}$ containing $z$. Declare $\mathcal{U}%
_{\alpha}\leq\mathcal{U}_{\beta}$ to mean that $\mathcal{U}_{\beta}$ refines
$\mathcal{U}_{\alpha}$. Using Lebesgue numbers, it is easy to see that
$\Omega$ is a directed set. For each $\mathcal{U}_{\alpha}$, let $N_{\alpha}$
be its nerve, and for each $\mathcal{U}_{\alpha}\leq\mathcal{U}_{\beta}$ let
$g_{\alpha}^{\beta}:N_{\beta}\rightarrow N_{\alpha}$ be (the pointed homotopy
class of) an induced simplicial map. In this way, we associate to $Z$ an
inverse system $\left\{  N_{\alpha},g_{\alpha}^{\beta};\Omega\right\}  $ from
$\operatorname*{pro}$-$\mathcal{FH}_{0}$. We may then define
$\operatorname*{pro}$-$\pi_{k}\left(  Z\right)  $ (the $k^{\text{th}}$
\emph{pro-homotopy group} \emph{of }$Z$) to be the inverse system $\left\{
\pi_{1}(N_{\alpha}),(g_{\alpha}^{\beta})_{\ast};\Omega\right\}  $ and
$\check{\pi}_{k}\left(  Z\right)  $ (the $k^{\text{th}}$ \emph{\v{C}ech
homotopy group} \emph{of }$Z$) its inverse limit.

Any cofinal tower contained in the above inverse system will be called a\emph{
tower of finite complexes associated to }$Z$. Another application of Lebesgue
numbers shows that such towers always exist. We say that $Z$ and $Z^{\prime}$
have the same \emph{pointed shape} if their associated towers are
pro-isomorphic. The \emph{shape dimension} of $Z$ is defined to be the
homotopy dimension of an associated tower. It is easy to see that the shape
dimension of $Z$ is less than or equal to its topological
dimension.\footnote{Another method for associating a tower of finite complexes
to $Z$ is to realize $Z$ as the inverse limit of such complexes. It is a
standard fact in shape theory that such a sequence will be pro-isomorphic to
the ones obtained above. See, for example, \cite{Bo} or \cite{MS}.}

Since associated towers $\left\{  N_{i},g_{i}\right\}  $ for $Z$ are, by
definition, cofinal subsystems of $\left\{  N_{\alpha},g_{\alpha}^{\beta
};\Omega\right\}  $, each comes with a canonical isomorphism%
\[
j:\check{\pi}_{k}\left(  \left\{  N_{i},g_{i}\right\}  \right)
=\underleftarrow{\lim}\left\{  \pi_{k}(N_{i}),(g_{i})_{\ast}\right\}
\rightarrow\underleftarrow{\lim}\left\{  \pi_{k}(N_{\alpha}),(g_{\alpha
}^{\beta})_{\ast};\Omega\right\}  \equiv\check{\pi}_{k}\left(  Z\right)  .
\]

\subsection{The reduced projective class group}

If $\Lambda$ is a ring, we say that two finitely generated projective
$\Lambda$-modules $P$ and $Q$ are \emph{stably equivalent} if there exist
finitely generated free $\Lambda$-modules $F_{1\text{ }}$ and $F_{2}$ such
that $P\oplus F_{1}\cong Q\oplus F_{2}$. Under the operation of direct sum,
the stable equivalence classes of finitely generated projective modules form a
group $\widetilde{K}_{0}\left(  \Lambda\right)  $, known as the \emph{reduced
projective class group} of $\Lambda$. In this group, a finitely generated
projective $\Lambda$-module $P$ represents the trivial element if and only if
it is \emph{stably free}, i.e., there exists a finitely generated free
$\Lambda$-module $F$ such that $P\oplus F$ is free.

Of particular interest is the case where $G$ is a group and $\Lambda$ is the
group ring $%
\mathbb{Z}
\left[  G\right]  $. Then $\widetilde{K}_{0}$ determines a functor from the
category $\mathcal{G}$ of groups to the category $\mathcal{AG}$ of abelian
groups. In particular, a group homomorphism $\lambda:G\rightarrow H$ induces a
ring homomorphism $%
\mathbb{Z}
\left[  G\right]  \rightarrow%
\mathbb{Z}
\left[  H\right]  $, which induces a group homomorphism $\lambda_{\ast
}:\widetilde{K}_{0}\left(
\mathbb{Z}
\left[  G\right]  \right)  \rightarrow\widetilde{K}_{0}\left(
\mathbb{Z}
\left[  H\right]  \right)  $.

\section{Main Results}

We are now ready to state and prove the main results of this paper.

\begin{theorem}
\label{main}Let $\left\{  K_{i},f_{i}\right\}  $ be a finite-dimensional tower
of pointed, connected, finite complexes having stable $\operatorname*{pro}%
$-$\pi_{k}$ for all $k$. Then there is a well defined obstruction
$\omega\left(  \left\{  K_{i},f_{i}\right\}  \right)  \in\widetilde{K}%
_{0}\left(
\mathbb{Z}
\mathbb{[}\check{\pi}_{1}\left(  \left\{  K_{i},f_{i}\right\}  )]\right)
\right)  $ which vanishes if and only if $\left\{  K_{i},f_{i}\right\}  $\ is
stable in $\operatorname*{pro}$-$\mathcal{FH}_{0}$.
\end{theorem}

Translating Theorem \ref{main} into the language of shape theory yields the
desired solution to Problem B:

\begin{theorem}
\label{shape}A connected compactum\ $Z$\ with finite shape dimension has the
pointed shape of a finite CW complex\ if and only if each of its homotopy
pro-groups is stable and an intrinsically defined Wall obstruction
$\omega\left(  Z\right)  \in\widetilde{K}_{0}\left(
\mathbb{Z}
\mathbb{[}\check{\pi}_{1}\left(  Z\right)  ]\right)  $ vanishes.\emph{\ }
\end{theorem}

Our proof of Theorem \ref{main} begins with two lemmas. The first is a simple
and well known algebraic observation.

\begin{lemma}
\label{fg}Let $C_{\ast}$ be a chain complex of finitely generated free
$\Lambda$-modules, and suppose that $H_{i}\left(  C_{\ast}\right)  =0$ for
$i\leq k$. Then

\begin{enumerate}
\item $\ker\partial_{i}$ is finitely generated and stably free for all $i\leq
k+1$, and

\item $H_{k+1}\left(  C_{\ast}\right)  $ is finitely generated.
\end{enumerate}
\end{lemma}

\begin{proof}
For the first assertion, begin by noting that $\ker\partial_{0}=C_{0}$ is
finitely generated and free. Proceeding inductively for $j\leq k+1,$ assume
that $\ker\partial_{j-1}$ is finitely generated and stably free. Since
$H_{j-1}\left(  C_{\ast}\right)  $ is trivial, we have a short exact sequence
\[
0\rightarrow\ker\partial_{j}\rightarrow C_{j}\rightarrow\ker\partial
_{j-1}\rightarrow0.
\]
By our assumption on $\ker\partial_{j-1}$, the sequence splits. Therefore,
$\ker\partial_{j}\oplus\ker\partial_{j-1}\cong C_{j}$, which implies that
$\ker\partial_{j}$ is finitely generated and stably free.

The second assertion follows from the first since $H_{k+1}\left(  C_{\ast
}\right)  =\ker\partial_{k+1}\diagup\operatorname*{im}\partial_{k+2}$.
\end{proof}

The second lemma---which is really the starting point to our proof of Theorem
\ref{main}---was extracted from \cite[Th. 4]{Fe1}. It uses the following
standard notation and terminology. For a map $f:K\rightarrow L$, the mapping
cylinder of $f$ will be denoted $M\left(  f\right)  $. The relative homotopy
and homology groups of the pair $(M\left(  f\right)  ,K)$ will be abbreviated
to $\pi_{i}(f)$ and $H_{i}\left(  f\right)  $. We say that $f$ is
$k$\emph{-connected} if $\pi_{i}(f)=0$ for all $i\leq k$; or equivalently,
$f_{\ast}:\pi_{i}(K)\rightarrow\pi_{i}(L)$ is an isomorphism for $i<k$ and a
surjection when $i=k$. The universal cover of a space $K$ will be denoted
$\widetilde{K}$. If $f:K\rightarrow L$ induces a $\pi_{1}$-isomorphism, then
$\widetilde{f}:\widetilde{K}\rightarrow\widetilde{L}$ denotes a lift of $f$.

\begin{lemma}
[The Tower Improvement Lemma]\label{ferry}Let $\left\{  K_{i},f_{i}\right\}  $
be a tower of pointed, connected, finite complexes with stable
$\operatorname*{pro}$-$\pi_{k}$ for $k\leq n$ and semistable
$\operatorname*{pro}$-$\pi_{n+1}$. Then there is a pro-isomorphic tower
$\left\{  L_{i},g_{i}\right\}  $ of finite complexes with the property that
each $g_{i}$ is $\left(  n+1\right)  $-connected. Moreover, after passing to a
subsequence of $\left\{  K_{i},f_{i}\right\}  $ and relabelling, we may assume that:

\begin{enumerate}
\item each $L_{i}$ is constructed from $K_{i}$ by inductively attaching
finitely many $k$-cells for $2\leq k\leq n+2$,

\item each $g_{i}$ is an extension of $f_{i}$ with $g_{i}\left(  K_{i}%
\cup\left(  \text{new cells of dimension }\leq k\right)  \right)  \subset$
$\left(  K_{i-1}\cup\left(  \text{new cells of dimension }\leq k-1\right)
\right)  $, and

\item the inclusions $K_{i}\hookrightarrow L_{i}$ form the promised
pro-isomorphism from $\left\{  K_{i},f_{i}\right\}  $ to $\left\{  L_{i}%
,g_{i}\right\}  $.
\end{enumerate}
\end{lemma}

\begin{proof}
Our proof is by induction on $n$.\medskip

\noindent\emph{Step 1. }($n=0$) Let $\left\{  K_{i},f_{i}\right\}  $ be a
tower with semistable $\operatorname*{pro}$-$\pi_{1}$. By attaching $2$-cells
to the $K_{i}$'s, we wish to obtain a new tower in which all bonding maps
induce surjections on $\pi_{1}$.

By semistability, we may (by passing to a subsequence and relabelling) assume
that each $f_{i\ast}$ maps $f_{i+1\ast}(\pi_{1}(K_{i+1}))$ onto $f_{i\ast}%
(\pi_{1}(K_{i}))$. Let $\left\{  ^{i}a_{j}\right\}  _{j=1}^{N_{i}}$ be a
finite generating set for $\pi_{1}(K_{i})$ and for each $^{i}a_{j}$ choose
$^{i}b_{j}\in f_{i+1\ast}(\pi_{1}(K_{i+1}))$ such that $f_{i\ast}\left(
^{i}a_{j}\right)  =f_{i\ast}\left(  ^{i}b_{j}\right)  $. For each element of
the form $^{i}a_{j}\left(  ^{i}b_{j}\right)  ^{-1}\in\pi_{1}(K_{i})$, attach a
$2$-cell to $K_{i}$ which kills that element. Call the resulting complexes
$L_{i}$'s, and note that each $f_{i}$ extends to a map $k_{i}:L_{i}\rightarrow
K_{i-1}$. Define $g_{i}:L_{i}\rightarrow L_{i-1}$ to be $k_{i}$ composed with
the inclusion $K_{i-1}\hookrightarrow L_{i-1}$. This leads to the following
commutative diagram:
\begin{equation}
\begin{diagram} K_{0}& & \lTo^{f_{1}} & & K_{1} & & \lTo^{f_{2}} & & K_{2} & & \lTo^{f_{3}} & & K_{3} &\cdots\\ & \luTo^{k_{1}} & & \ldInto & & \luTo^{k_{2}} & & \ldInto & & \luTo^{k_{3}} & & \ldInto & \\ & & L_{1} & & \lTo^{g_{2}} & & L_{2} & & \lTo^{g_{3}} & & L_{3} & & \lTo^{g_{4}} & &\cdots & \\ \end{diagram} \label{ladder with inclusions}%
\end{equation}
which ensures that the tower $\left\{  L_{i},g_{i}\right\}  $ is
pro-isomorphic to the original via inclusions.

Note that each $g_{i+1\ast}:\pi_{1}(L_{i+1})\rightarrow\pi_{1}(L_{i})$ is
surjective. Indeed, the loops in $K_{i}$ corresponding to the generating set
$\{^{i}a_{j}\}$ of $\pi_{1}(K_{i})$ still generate $\pi_{1}(L_{i})$; moreover,
in $\pi_{1}(L_{i})$ each $^{i}a_{j}\ $becomes identified with $^{i}b_{j}$
which lies in $\operatorname*{im}\left(  g_{i+1\ast}\right)  $. Properties 1
and 2 are immediate from the construction.\medskip

\noindent\emph{Step 2. }($n>0$) Now suppose $\left\{  K_{i},f_{i}\right\}  $
is a tower such that $\operatorname*{pro}$-$\pi_{k}$ is stable for all $k\leq
n$ and $\operatorname*{pro}$-$\pi_{n+1}$ is semistable.

We may assume inductively that there is a tower $\left\{  L_{i}^{\prime}%
,g_{i}^{\prime}\right\}  $ which has $n$-connected bonding maps and (after
passing to a subsequence of $\left\{  K_{i},f_{i}\right\}  $ and relabelling) satisfies:

\begin{enumerate}
\item[1$^{\prime}$.] each $L_{i}^{\prime}$ is constructed from $K_{i}$ by
inductively attaching finitely many $k$-cells for $2\leq k\leq n+1$, and

\item[2$^{\prime}$.] each $g_{i}^{\prime}$ is an extension of $f_{i}$ such
that $g_{i}^{\prime}\left(  K_{i}\cup\left(  \text{new cells of dimension
}\leq k\right)  \right)  \subset$ $\left(  K_{i-1}\cup\left(  \text{new cells
of dimension }\leq k-1\right)  \right)  $.

\item[3$^{\prime}$.] $\left\{  L_{i}^{\prime},g_{i}^{\prime}\right\}  $ and
$\left\{  K_{i},f_{i}\right\}  $ are pro-isomorphic via inclusions.
\end{enumerate}

\noindent Since $\operatorname*{pro}$-$\pi_{n+1}$ is semistable, we may also
assume that:

\begin{enumerate}
\item[4$^{\prime}$.] $g_{i\ast}^{\prime}$ maps $g_{i+1\ast}^{\prime}(\pi
_{n+1}(L_{i+1}^{\prime}))$ onto $g_{i\ast}^{\prime}(\pi_{n+1}(L_{i}^{\prime
}))$ for all $i$.
\end{enumerate}

Since the $g_{i}^{\prime}$'s are $n$-connected, then each $g_{i\ast}^{\prime
}:\pi_{k}\left(  L_{i}^{\prime}\right)  \rightarrow\pi_{k}\left(
L_{i-1}^{\prime}\right)  $ is an isomorphism for $k<n$. In addition, each
$g_{i\ast}^{\prime}:\pi_{n}\left(  L_{i}^{\prime}\right)  \rightarrow\pi
_{n}\left(  L_{i-1}^{\prime}\right)  $ is surjective; but since
$\operatorname*{pro}$-$\pi_{n}$ is stable, all but finitely many of these
surjections must be isomorphisms. So, by dropping finitely many terms and
relabelling, we assume that these also are isomorphisms.

Our goal is now clear---by attaching $\left(  n+2\right)  $-cells to the
$L_{i}^{\prime}$'s, we wish to make each bonding map $\left(  n+1\right)  $-connected.

Due to the $\pi_{n}$-isomorphisms just established, we have an exact sequence
\begin{equation}
\cdots\rightarrow\pi_{n+1}\left(  L_{i}^{\prime}\right)  \overset{g_{i\ast
}^{\prime}}{\longrightarrow}\pi_{n+1}\left(  L_{i-1}^{\prime}\right)
\rightarrow\pi_{n+1}\left(  g_{i}^{\prime}\right)  \rightarrow0,
\label{sequence 1}%
\end{equation}
for each $i$. Furthermore, since $n\geq1$, each $g_{i}^{\prime}$ induces a
$\pi_{1}$-isomorphism, so we may pass to the universal covers to obtain (by
covering space theory and the Hurewicz theorem) isomorphisms:
\begin{equation}
\pi_{n+1}\left(  g_{i}^{\prime}\right)  \cong\pi_{n+1}\left(  \widetilde
{g}_{i}^{\prime}\right)  \cong H_{n+1}\left(  \widetilde{g}_{i}^{\prime
}\right)  . \label{sequence 2}%
\end{equation}
Each term in the cellular chain complex $C_{\ast}\left(  \widetilde{g}%
_{i}^{\prime}\right)  $ is a finitely generated $%
\mathbb{Z}
\mathbb{[\pi}_{1}\left(  L_{i}\right)  ]$-module; so, by Lemma \ref{fg},
$H_{n+1}\left(  \widetilde{g}_{i}^{\prime}\right)  $ is finitely generated.

Applying ($\ref{sequence 2}$), we may choose a finite generating set $\left\{
^{i}\bar{\alpha}_{j}\right\}  _{j=1}^{N_{i}}$ for each $\pi_{n+1}\left(
g_{i}^{\prime}\right)  $; and by ($\ref{sequence 1}$), each $^{i}\bar{\alpha
}_{j}$ may be represented by an $^{i}\alpha_{j}^{\prime}\in\pi_{n+1}\left(
L_{i-1}^{\prime}\right)  $. By Condition 3$^{\prime}$ we may choose for each
$^{i}\alpha_{j}^{\prime}$, some $^{i}\beta_{j}\in\pi_{n+1}\left(
L_{i}^{\prime}\right)  $ such that $g_{i-1}^{\prime}\circ g_{i}^{\prime
}\left(  ^{i}\beta_{j}\right)  =g_{i-1}^{\prime}\left(  ^{i}\alpha_{j}%
^{\prime}\right)  $. Let $^{i}\alpha_{j}=\,^{i}\alpha_{j}^{\prime}%
-g_{i}^{\prime}\left(  ^{i}\beta_{j}\right)  \in\pi_{n+1}\left(
L_{i-1}^{\prime}\right)  $. Then each $^{i}\alpha_{j}$ is sent to $^{i}%
\bar{\alpha}_{j}$ in $\pi_{n+1}\left(  g_{i}^{\prime}\right)  $ and
$g_{i-1\ast}^{\prime}\left(  ^{i}\alpha_{j}\right)  =0\in\pi_{n+1}\left(
L_{i-2}^{\prime}\right)  $. Attach $\left(  n+2\right)  $-cells to each
$L_{i-1}^{\prime}$ to kill the $^{i}\alpha_{j}$'s. Call the resulting
complexes $L_{i}$'s, and for each $i$ let $k_{i}:L_{i}\rightarrow
L_{i-1}^{\prime}$ be an extension of $g_{i}^{\prime}$. Then let $g_{i}%
:L_{i}\rightarrow L_{i-1}$ be the composition of $k_{i}$ with the inclusion
$L_{i-1}^{\prime}\hookrightarrow L_{i-1}$. This leads to a diagram like that
produced in Step 1, hence the new system $\left\{  L_{i},g_{i}\right\}  $ is
pro-isomorphic to $\left\{  L_{i}^{\prime},g_{i}^{\prime}\right\}  $, and thus
to $\left\{  K_{i},f_{i}\right\}  $ via inclusions. Moreover, it is easy to
check that each $g_{i}$ is $\left(  n+1\right)  $-connected. Properties 1 and
2 are immediate from the construction and the inductive hypothesis, and
Property 3 from the final ladder diagram.
\end{proof}

Suppose now that $\left\{  K_{i},f_{i}\right\}  $ has stable
$\operatorname*{pro}$-$\pi_{k}$ for all $k$. Then, by repeatedly attaching
cells to the $K_{i}$'s, one may obtain pro-isomorphic towers with
$r$-connected bonding maps for arbitrarily large $r$. If $\left\{  K_{i}%
,f_{i}\right\}  $ is finite-dimensional it seems reasonable that, once $r$
exceeds the dimension of $\left\{  K_{i},f_{i}\right\}  $, this procedure will
terminate with bonding maps that are connected in all dimensions---and thus,
homotopy equivalences. Unfortunately, this strategy is too simplistic---in
order to obtain $r$-connected maps we must attach $\left(  r+1\right)
$-cells; thus, the dimensions of the complexes continually exceeds the
connectivity of the bonding maps. Roughly speaking, Theorem \ref{main}
captures the obstruction to making this strategy work.\bigskip

\begin{proof}
[Proof of Theorem \ref{main}]Begin with a tower $\left\{  L_{i},g_{i}\right\}
$ of $q$-dimensional complexes pro-isomorphic to $\left\{  K_{i}%
,f_{i}\right\}  $, via a diagram of type (\ref{ladder with inclusions}), which
has the following properties for all $i$.\medskip

\begin{itemize}
\item[a)] $g_{i}$ is $\left(  q-1\right)  $-connected,

\item[b)] for $k\in\left\{  q-2,q-1,q\right\}  $, $g_{i}$ maps the
$k$-skeleton of $L_{i}$ into the $\left(  k-1\right)  $-skeleton of $L_{i-1}$,

\item[c)] $g_{i\ast}$ maps $g_{i+1\ast}(\pi_{q}(L_{i+1}))$ onto $g_{i\ast}%
(\pi_{q}(L_{i}))$, and

\item[d)] $g_{i\ast}:\pi_{q-1}\left(  L_{i}\right)  \rightarrow\pi
_{q-1}\left(  L_{i-1}\right)  $ is an isomorphism.\medskip
\end{itemize}

\noindent A tower satisfying Conditions a) and b) is easily obtainable; apply
Lemma \ref{ferry} to $\left\{  K_{i},f_{i}\right\}  $ with $n=\dim\left\{
K_{i},f_{i}\right\}  +1$, in which case $q=\dim\left\{  K_{i},f_{i}\right\}
+3$. (\textbf{Note.} Although it may seem excessive to allow $\dim\left\{
L_{i},g_{i}\right\}  $ exceed $\dim\left\{  K_{i},f_{i}\right\}  $ by $3$,
this is done to obtain Condition b), which is key to our argument.)
Semistability of $\operatorname*{pro}$-$\pi_{q}$ gives Condition c)---after
passing to a subsequence and relabeling. Then, since $\operatorname*{pro}%
$-$\pi_{q-1}$ is stable and each $g_{i\ast}:\pi_{q-1}\left(  L_{i}\right)
\rightarrow\pi_{q-1}\left(  L_{i-1}\right)  $ is surjective, we may drop
finitely many terms to obtain Condition d).

As in the proof of Lemma \ref{ferry}, $\pi_{q}\left(  g_{i}\right)  $ and
$H_{q}\left(  \widetilde{g}_{i}\right)  $ are isomorphic finitely generated $%
\mathbb{Z}
\left[  \pi_{1}L_{i}\right]  $-modules. We will show that, for all $i$,
$H_{q}\left(  \widetilde{g}_{i}\right)  $ is projective and that all of these
modules are stably equivalent. (This is a pleasant surprise, since the $L_{i}%
$'s and $g_{i}$'s may all be different.) Thus we obtain corresponding elements
$\left[  H_{q}\left(  \widetilde{g}_{i}\right)  \right]  $ of $\widetilde
{K}_{0}\left(
\mathbb{Z}
\mathbb{[}\pi_{1}\left(  L_{i}\right)  ]\right)  $. When these elements are
trivial, i.e., when the modules are stably free, we will show that, by
attaching finitely many $\left(  q+1\right)  $-cells to each $L_{i}$, bonding
maps can be made homotopy equivalences. To complete the proof we define a
single obstruction $\omega\left(  \left\{  K_{i},f_{i}\right\}  \right)  $ to
be the image of $\left(  -1\right)  ^{q+1}[H_{q}\left(  \widetilde{g}%
_{i}\right)  ]$ in $\widetilde{K}_{0}\left(
\mathbb{Z}
\mathbb{[}\check{\pi}_{1}\left(  \left\{  K_{i},f_{i}\right\}  \right)
]\right)  $ and show that this element is uniquely determined by $\left\{
K_{i},f_{i}\right\}  $.\medskip

\noindent\textbf{Notes.} \textbf{1)} To be more precise, the $H_{q}\left(
\widetilde{g}_{i}\right)  $ determine elements $\left(  -1\right)
^{q+1}[H_{q}\left(  \widetilde{g}_{i}\right)  ]$ of $\widetilde{K}_{0}\left(
\mathbb{Z}
\mathbb{[}\pi_{1}\left(  L_{i}\right)  ]\right)  $ which may be associated,
via inclusion maps (that induce $\pi_{1}$-isomorphisms), to elements of
$\widetilde{K}_{0}\left(
\mathbb{Z}
\mathbb{[}\pi_{1}\left(  K_{i}\right)  ]\right)  $, which in turn determine a
common element of $\widetilde{K}_{0}\left(
\mathbb{Z}
\mathbb{[}\check{\pi}_{1}\left(  \left\{  K_{i},f_{i}\right\}  \right)
]\right)  $ via the projection maps---which in our setting are all isomorphisms.

\noindent\textbf{2) }We have used a factor $\left(  -1\right)  ^{q+1}$
(instead of the more concise $\left(  -1\right)  ^{q}$) so that our definition
agrees with those already in the literature.\medskip

While most of our work takes place in the individual mapping cylinders
$M\left(  g_{i}\right)  $ and their universal covers, there is some interplay
between adjacent cylinders. For that reason, it is useful to view our work as
taking place in the \textquotedblleft infinite mapping
telescope\textquotedblright\ shown in Figure 1 (and in its universal cover).
\begin{figure}[ptb]
\begin{center}
\includegraphics[
height=1.00in,
width=4.00in
]{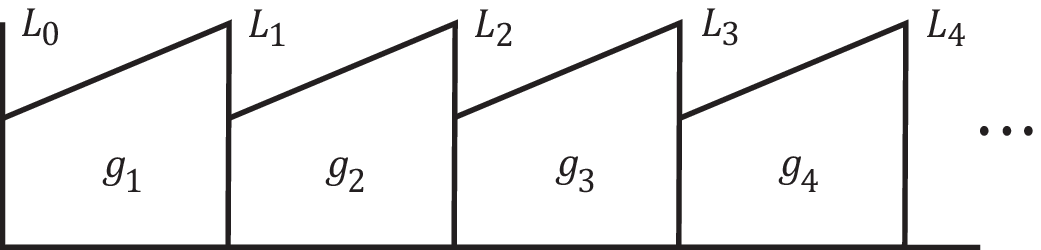}
\end{center}
\caption{\ }%
\label{Fig 1}%
\end{figure}

For ease of notation, fix $i$ and consider the pair $\left(  M\left(
\widetilde{g}_{i}\right)  ,\widetilde{L}_{i}\right)  $. It is a standard fact
(see \cite[3.9]{Co}) that $C_{\ast}\left(  \widetilde{g}_{i}\right)  $ is
isomorphic to the algebraic mapping cone\ of the chain homomorphism $g_{i\ast
}:C_{\ast}\left(  \widetilde{L}_{i}\right)  \rightarrow C_{\ast}\left(
\widetilde{L}_{i-1}\right)  $. In particular, if the cellular chain complexes
$C_{\ast}\left(  \widetilde{L}_{i-1}\right)  $ and $C_{\ast}\left(
\widetilde{L}_{i}\right)  $ of $\widetilde{L}_{i-1}$ and $\widetilde{L}_{i}$
are expressed as:
\begin{align}
0 &  \rightarrow D_{q}\overset{d_{q}}{\longrightarrow}D_{q-1}\overset{d_{q-1}%
}{\longrightarrow}\cdots\overset{d_{2}}{\longrightarrow}D_{1}\overset{d_{1}%
}{\longrightarrow}D_{0}\longrightarrow0\text{,\quad and\smallskip
}\label{sequence 3}\\
0 &  \rightarrow D_{q}^{\prime}\overset{d_{q}^{\prime}}{\longrightarrow
}D_{q-1}^{\prime}\overset{d_{q-1}^{\prime}}{\longrightarrow}\cdots
\overset{d_{2}^{\prime}}{\longrightarrow}D_{1}^{\prime}\overset{d_{1}^{\prime
}}{\longrightarrow}D_{0}^{\prime}\longrightarrow0,\label{sequence 4}%
\end{align}
respectively, then $C_{\ast}\left(  \widetilde{g}_{i}\right)  $ is naturally
isomorphic to a chain complex
\[
0\rightarrow C_{q+1}\overset{\partial_{q+1}}{\longrightarrow}C_{q}%
\overset{\partial_{q}}{\longrightarrow}\cdots\overset{\partial_{2}%
}{\longrightarrow}C_{1}\overset{\partial_{1}}{\longrightarrow}C_{0}%
\longrightarrow0
\]
where, for each $j$,
\[
C_{j}=D_{j-1}^{\prime}\oplus D_{j}\text{ \ and \ }\partial_{j}\left(
x,y\right)  =(-\,d_{j-1}^{\prime}x,\ \widetilde{g}_{i\ast}x+d_{j}y).
\]
Here one views each $\pi_{1}\left(  L_{i-1}\right)  $-module $D_{j}$ as a
$\pi_{1}\left(  L_{i}\right)  $-module in the obvious way---associating
$a\cdot x$ with $\widetilde{g}_{i\ast}\left(  a\right)  \cdot x$ for $a\in
\pi_{1}\left(  L_{i}\right)  $.

By Condition b), the map $\widetilde{g}_{i\ast}:D_{j}^{\prime}\rightarrow
D_{j}$ is trivial for $j\geq q-2$; so, in these dimensions, $\partial_{j}$
splits as $-\,d_{j-1}^{\prime}\oplus d_{j}$, allowing our chain complex to be
written:
\[
0\rightarrow\underset{C_{q+1}}{\underbrace{D_{q}^{\prime}\oplus0}}%
\overset{-\,d_{q}^{\prime}\oplus0}{\longrightarrow}\underset{C_{q}%
}{\underbrace{D_{q-1}^{\prime}\oplus D_{q}}}\overset{-\,d_{q-1}^{\prime}\oplus
d_{q}}{\longrightarrow}\underset{C_{q-1}}{\underbrace{D_{q-2}^{\prime}\oplus
D_{q-1}}}\overset{-\,d_{q-2}^{\prime}\oplus d_{q-1}}{\longrightarrow}%
\underset{C_{q-2}}{\underbrace{D_{q-3}^{\prime}\oplus D_{q-2}}}\overset
{\partial_{q-2}}{\longrightarrow}\cdots.
\]
Since the \textquotedblleft minus signs\textquotedblright\ have no effect on
kernels or images of maps, it follows that
\begin{align}
\ker\partial_{q-1}  &  =\ker(d_{q-2}^{\prime})\oplus\ker(d_{q-1})\label{1}\\
\ker\partial_{q}  &  =\ker(d_{q-1}^{\prime})\oplus\ker(d_{q})\label{2}\\
\ker\partial_{q+1}  &  =\ker(d_{q}^{\prime})\label{3}\\
H_{q-1}(\widetilde{g}_{i})  &  =\left(  \ker(d_{q-2}^{\prime})\diagup
\operatorname*{im}\left(  \,d_{q-1}^{\prime}\right)  \right)  \oplus\left(
\ker(d_{q-1})\diagup\operatorname*{im}\left(  \,d_{q}\right)  \right)
\label{4}\\
H_{q}(\widetilde{g}_{i})  &  =\left(  \ker(d_{q-1}^{\prime})\diagup
\operatorname*{im}\left(  d_{q}^{\prime}\right)  \right)  \oplus\ker
(d_{q})\label{5}\\
H_{q+1}(\widetilde{g}_{i})  &  =\ker(d_{q}^{\prime}) \label{6}%
\end{align}
Since $H_{q-1}(\widetilde{g}_{i})=0$, each summand in Identity \ref{4} is
trivial. Furthermore, the same reasoning applied to the adjacent mapping
cylinder $M\left(  g_{i+1}\right)  $ yields an analogous set of identities for
$C_{\ast}\left(  \widetilde{g}_{i+1}\right)  $ in which the \textquotedblleft
primed terms\textquotedblright\ become the \textquotedblleft unprimed
terms\textquotedblright. This shows that $\ker(d_{q-1}^{\prime})\diagup
\operatorname*{im}\left(  \,d_{q}^{\prime}\right)  $ is also trivial. Hence,
the first summand in Identity \ref{5} is trivial, so $H_{q}(\widetilde{g}%
_{i})\cong\ker(d_{q})$. Identity \ref{2} together with Lemma \ref{fg} then
shows that $H_{q}(\widetilde{g}_{i})$ is finitely generated and projective.
The same reasoning in $C_{\ast}\left(  \widetilde{g}_{i+1}\right)  $ shows
that $H_{q}(\widetilde{g}_{i+1})\cong\ker(d_{q}^{\prime})$ is finitely
generated projective, so by Identity \ref{6}, $H_{q+1}(\widetilde{g}_{i})$ is
finitely generated projective and naturally isomorphic to $H_{q}(\widetilde
{g}_{i+1})$ (using $\widetilde{g}_{i+1\ast}$ to make $H_{q+1}(\widetilde
{g}_{i})$ a $\pi_{1}\left(  L_{i+1}\right)  $-module).

Next we show that $H_{q}(\widetilde{g}_{i})$ and $H_{q+1}(\widetilde{g}_{i})$
are stably equivalent. Extract the short exact sequence
\[
0\rightarrow H_{q+1}(\widetilde{g}_{i})\rightarrow D_{q}^{\prime}%
\rightarrow\operatorname*{im}\left(  d_{q}^{\prime}\right)  \rightarrow0
\]
from above, then recall that $\operatorname*{im}\left(  d_{q}^{\prime}\right)
$ is equal to $\ker(d_{q-1}^{\prime})$. The latter is projective, so
\[
D_{q}^{\prime}\cong H_{q+1}(\widetilde{g}_{i})\oplus\ker(d_{q-1}^{\prime}).
\]
Thus $\left[  H_{q+1}(\widetilde{g}_{i})\right]  =-\left[  \ker(d_{q-1}%
^{\prime})\right]  $ in $\widetilde{K}_{0}\left(  \pi_{1}(L_{i}\right)  )$.
Since $\left[  H_{q}(\widetilde{g}_{i})\right]  =\left[  \ker(d_{q})\right]  $
and (by Identity \ref{2}), $\left[  \ker(d_{q})\right]  =-\left[  \ker
(d_{q-1}^{\prime})\right]  $, $\left[  H_{q}(\widetilde{g}_{i})\right]
=\left[  H_{q+1}(\widetilde{g}_{i})\right]  $.

To summarize, we have shown that for each $i$:\medskip

\begin{itemize}
\item $H_{q}(\widetilde{g}_{i})$ and $H_{q+1}(\widetilde{g}_{i})$ are finitely
generated and projective,

\item $\left[  H_{q}(\widetilde{g}_{i})\right]  =\left[  H_{q+1}(\widetilde
{g}_{i})\right]  $ in $\widetilde{K}_{0}\left(  \pi_{1}\left(  L_{i}\right)
\right)  $, and

\item $H_{q+1}(\widetilde{g}_{i})$ naturally isomorphic to $H_{q}%
(\widetilde{g}_{i+1})$ as $\pi_{1}\left(  L_{i+1}\right)  $-modules.\medskip
\end{itemize}

\noindent These observations combine to show that each $\left[  H_{q}%
(\widetilde{g}_{i})\right]  $ determines the \textquotedblleft
same\textquotedblright\ element of $\widetilde{K}_{0}\left(  \check{\pi}%
_{1}\left(  \left\{  K_{i},f_{i}\right\}  \right)  \right)  $. More precisely,
define $\omega\left(  \left\{  K_{i},f_{i}\right\}  \right)  $ to be the image
of $\left(  -1\right)  ^{q+1}[H_{q}\left(  \widetilde{g}_{i}\right)  ]$ under
the isomorphism $\widetilde{K}_{0}\left(
\mathbb{Z}
\mathbb{[}\pi_{1}\left(  L_{i}\right)  ]\right)  \rightarrow\widetilde{K}%
_{0}\left(  \check{\pi}_{1}\left(  \left\{  K_{i},f_{i}\right\}  \right)
\right)  $ induced by the composition of group isomorphisms%
\begin{equation}
\pi_{1}\left(  L_{i}\right)  \overset{p_{i}^{-1}}{\longrightarrow
}\underleftarrow{\lim}\left\{  \pi_{i}(L_{i}),(g_{i})_{\ast}\right\}
\rightarrow\underleftarrow{\lim}\left\{  \pi_{i}(K_{i}),(f_{i})_{\ast
}\right\}  =\check{\pi}_{1}\left(  \left\{  K_{i},f_{i}\right\}  \right)
\label{composition of isomorphisms}%
\end{equation}
where $p_{i}:\underleftarrow{\lim}\left\{  \pi_{i}(L_{i}),(g_{i})_{\ast
}\right\}  \rightarrow\pi_{1}\left(  L_{i}\right)  $ is projection, and the
isomorphism between inverse limits is canonically induced by ladder diagram
(\ref{ladder with inclusions}).\medskip

\noindent\emph{Claim 1.} If $\omega\left(  \left\{  K_{i},f_{i}\right\}
\right)  =0$, then $\left\{  K_{i},f_{i}\right\}  $ is stable.\smallskip

We will show that, by adding finitely many$q$- and $\left(  q+1\right)
$-cells to each of the above $L_{i}$'s, we may arrive at a pro-isomorphic
tower in which all bonding maps are homotopy equivalences.

By assumption, each $%
\mathbb{Z}
\mathbb{\pi}_{1}$-module $H_{q}(\widetilde{g}_{i})$ becomes free upon
summation with a finitely generated free module. This may be accomplished
geometrically by attaching finitely many $q$-cells to the corresponding
$L_{i-1}$'s via trivial attaching maps at the basepoints. Each $g_{i-1}$ may
then be extended by mapping these $q$-cells to the basepoint of $L_{i-2}$.
Since this procedure preserves all relevant properties of our tower, we will
assume that, for each $i$, $H_{q}(\widetilde{g}_{i})$ (and therefore $\pi
_{q}\left(  g_{i}\right)  $) is a finitely generated free $%
\mathbb{Z}
\mathbb{[\pi}_{1}\left(  L_{i}\right)  ]$-module.

Proceed as in Step 2 of the proof of Lemma \ref{ferry} to obtain collections
$\left\{  ^{i}\alpha_{j}\right\}  _{j=1}^{N_{i}}\subset\pi_{n+1}\left(
L_{i-1}\right)  $ that correspond to generating sets for the $\pi_{q}\left(
g_{i}\right)  $'s and which satisfy $g_{i-1\ast}\left(  ^{i}\alpha_{j}\right)
=0\in\pi_{q}\left(  L_{i-2}\right)  $ for all $i,j$. In addition, we now
require that $\left\{  ^{i}\alpha_{j}\right\}  _{j=1}^{N_{i}}$ corresponds to
a free basis for $\pi_{q}\left(  g_{i}\right)  $. For each $^{i}\alpha_{j}$
attach a single $\left(  q+1\right)  $-cell to $L_{i-1}$ to kill that element.
Extend each $g_{i}$ to $g_{i}^{\prime}:L_{i}^{\prime}\rightarrow
L_{i-1}^{\prime}$ as before, thereby obtaining a tower $\left\{  L_{i}%
^{\prime},g_{i}^{\prime}\right\}  $ for which all bonding maps are
$q$-connected. Since the $\left(  q+1\right)  $-cells are attached to
$L_{i-1}$ along a free basis, we do not create any new $\left(  q+1\right)
$-cycles for the pair $\left(  M\left(  \widetilde{g}_{i}^{\prime}\right)
,\widetilde{L}_{i}^{\prime}\right)  $, so no new $\left(  q+1\right)
$-dimensional homology is introduced. Moreover, the $\left(  q+1\right)
$-cells attached to $L_{i-1}$ result in $\left(  q+2\right)  $-cells in
$M(\widetilde{g}_{i-1}^{\prime})$ which are attached in precisely the correct
manner to kill $H_{q+1}\left(  \widetilde{g}_{i-1}\right)  $ without creating
any $\left(  q+2\right)  $-dimensional homology---this is due to the natural
isomorphism discovered earlier between $H_{q+1}\left(  \widetilde{g}%
_{i-1}\right)  $ and $H_{q}\left(  \widetilde{g}_{i}\right)  $. Thus the
$g_{i}^{\prime}$'s are all $\left(  n+2\right)  $-connected, and since the
$L_{i}^{\prime}$'s are $\left(  n+1\right)  $-dimensional, this means that the
$g_{i}^{\prime}$'s are homotopy equivalences. So $\left\{  L_{i}^{\prime
},g_{i}^{\prime}\right\}  $ and hence $\left\{  K_{i},f_{i}\right\}  $, are
stable in $\operatorname*{pro}$-$\mathcal{FH}_{0}$.\medskip

\noindent\emph{Claim 2.} The obstruction is well defined.\smallskip

We must show that $\omega(\left\{  K_{i},f_{i}\right\}  )$ does not depend on
the tower $\left\{  L_{i},g_{i}\right\}  $ and ladder diagram chosen at the
beginning of the proof. First observe that any subsequence $\left\{  L_{k_{i}%
},g_{k_{i}k_{i-1}}\right\}  $ of $\left\{  L_{i},g_{i}\right\}  $ yields the
same obstruction. This is immediate in the special case that $\left\{
L_{k_{i}},g_{k_{i}k_{i-1}}\right\}  $ contains two consecutive terms of
$\left\{  L_{i},g_{i}\right\}  $. If not, notice that $\left\{  L_{k_{i}%
},g_{k_{i}k_{i-1}}\right\}  $ is a subsequence of $L_{k_{1}}\leftarrow
L_{k_{1}+1}\leftarrow L_{k_{2}}\leftarrow L_{k_{3}}\leftarrow\cdots,$ which is
a subsequence of $\left\{  L_{i},g_{i}\right\}  $. Therefore the more general
observation follows from the special case.

Next suppose that $\left\{  L_{i},g_{i}\right\}  $ and $\left\{  M_{i}%
,h_{i}\right\}  $ are each towers of finite $q$-dimensional complexes
satisfying the conditions laid out at the beginning of the proof. Then
$\left\{  L_{i},g_{i}\right\}  $ and $\left\{  M_{i},h_{i}\right\}  $ are
pro-isomorphic; so, after passing to subsequences and relabeling, there exists
a homotopy commuting diagram of the form:
\[
\begin{diagram}
L_{0}& & \lTo^{g_{1}} & & L_{1} & & \lTo^{g_{2}} & &
L_{2} & & \lTo^{g_{3}} & & L_{3} &\cdots\\
& \luTo^{{\lambda}_{1}} & & \ldTo^{{\mu}_{1}} & & \luTo^{{\lambda}_{2}} & &  \ldTo^{{\mu}_{2}}   & & \luTo^{{\lambda}_{3}} & & \ldTo^{{\mu}_{3}} & \\
& & M_{1} & & \lTo^{h_{2}} & & M_{2} & & \lTo^{h_{3}}  & & M_{3} & & \lTo^{h_{4}} & & \cdots & \\
\end{diagram}
\]
where all $\lambda_{i}$ and $\mu_{i}$ are cellular maps. From here we may
create a new tower:
\[
M_{1}\longleftarrow L_{2}\longleftarrow M_{4}\longleftarrow L_{5}%
\longleftarrow M_{7}\longleftarrow L_{8}\longleftarrow M_{10}\longleftarrow
\cdots
\]
where the bonding maps are determined (up to homotopy) by the ladder diagram.
Properties a),c) and d) hold for this tower due to the corresponding
properties for $\left\{  L_{i},g_{i}\right\}  $ and $\left\{  M_{i}%
,h_{i}\right\}  $. To see that Property b) holds, note that each bonding map
is the composition of a $g_{i}$ or an $h_{i}$ with a cellular map. (This is
why so many terms were omitted.) Since this new tower contains subsequences
which are---up to homotopies of the bonding maps---subsequences of $\left\{
L_{i},g_{i}\right\}  $ and $\left\{  M_{i},h_{i}\right\}  $, our initial
observation implies that all determine the same obstruction.

Finally we consider the general situation where $\left\{  L_{i},g_{i}\right\}
$ and $\left\{  M_{i},h_{i}\right\}  $ satisfy Conditions a)-d), but are not
necessarily of the same dimension. By the previous case and induction, it will
be enough to show that, for a given $q$-dimensional $\left\{  L_{i}%
,g_{i}\right\}  $, we can find a $\left(  q+1\right)  $-dimensional tower
$\left\{  L_{i}^{\prime},g_{i}^{\prime}\right\}  $ which satisfies the
corresponding versions of Conditions a)-d), and which determines the same
obstruction as $\left\{  L_{i},g_{i}\right\}  $. In this step, the need for
the $\left(  -1\right)  ^{q+1}$ factor finally becomes clear.

The tower $\left\{  L_{i}^{\prime},g_{i}^{\prime}\right\}  $ is obtained by
carrying out our usual strategy of attaching a finite collection of $\left(
q+1\right)  $-cells to each $L_{i-1}$ along a generating set for $H_{q}\left(
M\left(  \widetilde{g}_{i}\right)  ,\widetilde{L}_{i}\right)  $. The resulting
$C_{\ast}\left(  \widetilde{L}_{i}^{\prime}\right)  $'s differ from the
$C_{\ast}\left(  \widetilde{L}_{i}\right)  $'s only in dimension $q+1$ where
we have introduced finitely generated free modules $^{i}F_{q+1}$. By inserting
this term into (\ref{sequence 3}) and rewriting $D_{q}$ as $\operatorname*{im}%
\left(  d_{q}\right)  \oplus\ker\left(  d_{q}\right)  ,$ the chain complex for
$L_{i-1}^{\prime}$ may be written:\smallskip\
\[
0\longrightarrow\,^{i}F_{q+1}\overset{d_{q+1}}{\longrightarrow}%
\operatorname*{im}\left(  d_{q}\right)  \oplus\ker\left(  d_{q}\right)
\overset{d_{q}}{\longrightarrow}D_{q-1}\overset{d_{q-1}}{\longrightarrow
\smallskip}\cdots\overset{d_{2}}{\longrightarrow}D_{1}\overset{d_{1}%
}{\longrightarrow}D_{0}\longrightarrow0,
\]
By construction, $d_{q+1}$ takes $\,^{i}F_{q+1}$ onto $\ker\left(
d_{q}\right)  $ thereby eliminating the $q$-dimensional homology of the pair
$\left(  M\left(  \widetilde{g}_{i}^{\prime}\right)  ,\widetilde{L}%
_{i}^{\prime}\right)  $. Note however, that we may have introduced new
$(q+1)$-dimensional homology. Indeed, by our earlier analysis, $H_{q+1}%
(\widetilde{g}_{i}^{\prime})=\ker(\,d_{q+1})$. (The original $\left(
q+1\right)  $-dimensional homology of the pair was eliminated---as it was in
the unobstructed case---when we attached $\left(  q+1\right)  $-cells to
$L_{i}$.) \ By extracting the short exact sequence
\[
0\longrightarrow\ker\left(  d_{q+1}\right)  \,\longrightarrow\,^{i}%
F_{q+1}\longrightarrow\ker\left(  d_{q}\right)  \longrightarrow0
\]
and recalling that $\ker\left(  d_{q}\right)  \cong H_{q}\left(  \widetilde
{g}_{i}\right)  $ is projective, we have$\,$%
\[
^{i}F_{q+1}\cong H_{q+1}(\widetilde{g}_{i}^{\prime})\oplus H_{q}\left(
\widetilde{g}_{i}\right)  .
\]
So, upon projection into $\widetilde{K}_{0}\left(  \check{\pi}_{1}\left(
\left\{  K_{i},f_{i}\right\}  \right)  \right)  $, (as described in line
(\ref{composition of isomorphisms})), $\left[  H_{q+1}(\widetilde{g}%
_{i}^{\prime})\right]  $ and $-\left[  H_{q}\left(  \widetilde{g}_{i}\right)
\right]  $ determine the same element. The same is then true for $\left(
-1\right)  ^{q+2}\left[  H_{q+1}(\widetilde{g}_{i}^{\prime})\right]  $ and
$\left(  -1\right)  ^{q+1}\left[  H_{q}\left(  \widetilde{g}_{i}\right)
\right]  $, showing that $\left\{  L_{i},g_{i}\right\}  $ and $\left\{
L_{i}^{\prime},g_{i}^{\prime}\right\}  $ lead to the same obstruction.
\end{proof}

\begin{proof}
[Proof of Theorem \ref{shape}]We need only verify the forward implication, as
the converse is obvious.

Using the finite-dimensionality of $Z$, choose a finite-dimensional tower of
pointed, connected, finite complexes $\left\{  N_{i},g_{i}\right\}  $
associated to $Z$. By the pro-$\pi_{k}$ hypotheses on $Z$, we may apply
Theorem \ref{main} to obtain $\omega\left(  \left\{  N_{i},g_{i}\right\}
\right)  \in\widetilde{K}_{0}\left(
\mathbb{Z}
\mathbb{[}\check{\pi}_{1}\left(  \left\{  N_{i},g_{i}\right\}  )]\right)
\right)  $. The inclusion of $\left\{  N_{i},g_{i}\right\}  $ into the
associated inverse system $\left\{  N_{\alpha},g_{\alpha}^{\beta}%
;\Omega\right\}  $, as described in \S \ref{Subsection: shapes of compacta},
yields a canonical isomorphism of $\check{\pi}_{1}\left(  \left\{  K_{i}%
,g_{i}\right\}  )\right)  $ onto $\check{\pi}_{1}\left(  \left\{  N_{\alpha
},g_{\alpha}^{\beta};\Omega\right\}  \right)  =\check{\pi}_{1}\left(
Z\right)  $ which converts $\omega\left(  \left\{  N_{i},g_{i}\right\}
\right)  $ to our intrinsically defined Wall obstruction $\omega\left(
Z\right)  \in\widetilde{K}_{0}\left(
\mathbb{Z}
\mathbb{[}\check{\pi}_{1}\left(  Z\right)  ]\right)  $.
\end{proof}

\section{Realizing the obstructions}

In addition to proving Theorems \ref{main} and \ref{shape}, Edwards and
Geoghegan showed how to build towers and compacta with non-trivial
obstructions. By applying their strategy within our framework, we obtain an
easy proof of the following:

\begin{proposition}
Let $G$ be a finitely presentable group and $P$ a finitely generated
projective $%
\mathbb{Z}
\left[  G\right]  $ module. Then there exists a tower of finite $2$-complexes
$\left\{  K_{i},f_{i}\right\}  $, with stable $\operatorname*{pro}$-$\pi_{k}$
for all $k$ and $\check{\pi}_{1}\left(  \left\{  K_{i},f_{i}\right\}  \right)
\cong G$, such that $\omega\left(  \left\{  K_{i},f_{i}\right\}  \right)
=\left[  P\right]  \in\widetilde{K}_{0}\left(
\mathbb{Z}
\left[  G\right]  \right)  $.
\end{proposition}

By letting $Z=\underleftarrow{\lim}\left\{  K_{i},f_{i}\right\}  $ we
immediately obtain:

\begin{proposition}
Let $G$ be a finitely presentable group and $P$ a finitely generated
projective $%
\mathbb{Z}
\left[  G\right]  $ module. Then there exists a compact connected
$2$-dimensional pointed compactum $Z$, with stable $\operatorname*{pro}$%
-$\pi_{k}$ for all $k$ and $\check{\pi}_{1}\left(  Z\right)  \cong G$, such
that $\omega\left(  Z\right)  =\left[  P\right]  \in\widetilde{K}_{0}\left(
\mathbb{Z}
\left[  G\right]  \right)  $.
\end{proposition}

\begin{proof}
Let $Q$ be a finitely generated projective $%
\mathbb{Z}
\left[  G\right]  $ module representing $-\left[  P\right]  $ in
$\widetilde{K}_{0}\left(
\mathbb{Z}
\left[  G\right]  \right)  $, and so that $F=P\oplus Q$ is finitely generated
and free. Let $r$ denote the rank of $F$. Let $K^{\prime}$ be a finite pointed
$2$-complex with $\pi_{1}\left(  K^{\prime}\right)  \cong G$, then construct
$K$ from $K^{\prime}$ by wedging a bouquet of $r$ $2$-spheres to $K^{\prime}$
at the basepoint. Then $\pi_{2}\left(  K\right)  \cong H_{2}\left(
\widetilde{K}\right)  $ has a summand isomorphic to $F$ which corresponds to
the bouquet of $2$-spheres. Define a map $f:K\rightarrow K$ so that
$f\mid_{K^{\prime}}=\operatorname*{id}$ and $f_{\ast}:\pi_{2}\left(  K\right)
\rightarrow\pi_{2}\left(  K\right)  $ (or equivalently $\widetilde{f}_{\ast
}:H_{2}\left(  \widetilde{K}\right)  \rightarrow H_{2}\left(  \widetilde
{K}\right)  $) is the projection $P\oplus Q\rightarrow P$ when restricted to
the $F$-factor. Note that $H_{2}\left(  \widetilde{f}\right)  \cong Q\cong
H_{3}\left(  \widetilde{f}\right)  $. Obtain the tower $\left\{  K_{i}%
,f_{i}\right\}  $ by letting $K_{i}=K$ for all $k\geq0$ and $f_{i}=f$ for all
$k\geq1$.

To calculate $\omega\left(  \left\{  K_{i},f_{i}\right\}  \right)  $ according
to the proof of Theorem \ref{main}, we must attach cells of dimensions $3,4$
and $5$ to each $K_{i}$ to obtain an equivalent tower $\left\{  L_{i}%
,g_{i}\right\}  $ satisfying Conditions a)-d) of the proof. As we saw in Claim
2 of Theorem \ref{main}, this procedure simply shifts homology to higher
dimensions. In particular, $[H_{5}\left(  \widetilde{g}_{i}\right)
]=-[H_{2}\left(  \widetilde{f}\right)  ]=\left[  P\right]  $, as desired.
\end{proof}

\end{document}